\documentclass[11pt,reqno]{amsart}
\usepackage[T1]{fontenc}
\usepackage{graphicx, amssymb, amsmath}
\usepackage{mathtools}
\usepackage{cite}
\usepackage{color}
\definecolor{MyLinkColor}{rgb}{0,0,0.4}

\newcommand{\R}{{\mathbb R}}
\newcommand{\E}{{\mathcal E}}

\newcommand{\N}{{\mathbb N}}

\newcommand{\kH}{\mathcal{H}}

\newcommand{\kL}{\mathcal{L}}

\newcommand{\wt}{\widetilde}

\newcommand{\ov}{\overline}
\newcommand{\p}{\partial}

\newcommand{\e}{\varepsilon}

\newcommand{\0}{\Omega}

\newcommand{\supp}{\mathop{\rm supp}\nolimits}
\newcommand{\lt}{{L_2(\0)}}
\newcommand{\ds}{\,\mathrm{d}s}
\newcommand{\dz}{\,\mathrm{d}z}
\newcommand{\dtau}{\,\mathrm{d}\tau}

\newtheorem{thm}{Theorem}[section]
\newtheorem{prop}[thm]{Proposition}
\newtheorem{lemma}[thm]{Lemma}

\theoremstyle{remark}

\numberwithin{equation}{section} 
\usepackage[colorlinks=true,linkcolor=MyLinkColor,citecolor=MyLinkColor]{hyperref}

\makeatletter
\@namedef{subjclassname@2020}{\textup{2020} Mathematics Subject Classification}
\makeatother

\begin{document}

\title[An asymptotic  model for atmospheric flows]{Weak and classical solutions to an asymptotic  model for atmospheric flows}

\thanks{}
\author{Bogdan-Vasile Matioc}
\address{Fakult\"at f\"ur Mathematik, Universit\"at Regensburg \\ D--93040 Regensburg, Deutschland}
\email{bogdan.matioc@ur.de}

\author{Luigi Roberti}
\address{Faculty of Mathematics, University of Vienna, Oskar--Morgenstern--Platz 1, 1090 Vienna, Austria}
\email{luigi.roberti@univie.ac.at}
\begin{abstract}
In this paper we study a recently derived mathematical model for nonlinear propagation of waves in the atmosphere,
for which  we establish  the local well-posedness in the setting of classical solutions. 
This is achieved by formulating the model as a quasilinear parabolic evolution problem in an appropriate functional analytic framework and by using abstract theory for such problems.
Moreover, for $L_2$-initial data, we construct  global weak solutions by employing a two-step approximation strategy based on a Galerkin scheme,
 where an equivalent formulation of the problem in terms of a new variable is used. 
 Compared to the original model, the latter has the advantage that the~${L_2}$-norm is a Liapunov functional.
\end{abstract}


\maketitle

\pagestyle{myheadings}
\markboth{\sc{B.-V.~Matioc  \& L. Roberti}}{\sc{An asymptotic  model for atmospheric flows}}

 \section{Introduction and main results}\label{Sec:0}

We consider the system
\begin{subequations}\label{PB}
\begin{equation}\label{EE}
\left.
\arraycolsep=1.4pt
\begin{array}{rcl}
u_t+u u_x+v u_y&=&\mu\Delta u +\alpha u+\beta v +K,\\[1ex]
u_x+v_y&=&0
\end{array}
\right\}\quad \text{in $\0,$ $t>0,$} 
\end{equation}
 in the unbounded horizontal strip
\[
\0\coloneqq\{z \coloneqq (x,y)\in\R^2\,: x\in\R,\, \, 0<y<1\},
\]
where the constants that appear in \eqref{EE} satisfy
\[
\mu \in(0,\infty)\qquad\text{and}\qquad \alpha,\,\beta\in\R.
\]
The equations \eqref{EE} are subject to the following boundary conditions:
\begin{equation}\label{BC}
\left.
\arraycolsep=1.4pt
\begin{array}{rcl}
 u&=&0\quad \text{on $\p\0$, \, $t>0$,} \\[1ex]
 v&=&0\quad \text{on $\{y=0\}$, \, $t>0$}.
\end{array}
\right\}
\end{equation}
Additionally, it is assumed that $u$ is known initially:
\begin{equation}\label{IC}
u(0)=u_0,
\end{equation}
where $u_0:\0\to\R$ is a given function.
\end{subequations}

The system \eqref{PB} has been only very recently derived in \cite[Equation (6.18)]{CJ22} as an asymptotic model of the general equations governing the atmospheric flow, and describes nonlinear wave propagation in the troposphere. In particular, it is meant to be a rigorous mathematical model of the \emph{morning glory cloud} pattern, a spectacular atmospheric phenomenon -- taking place in coastal regions of particular shape, especially in Australia -- whereby a train of long, narrow tubular clouds propagates for hundreds of kilometres in the horizontal direction perpendicular to the cloud line; see, for instance, \cite{Cr92,BiRe13,SmRoCr82}. Thus the motion is essentially two-dimensional, i.e., in the horizontal and vertical directions, which we denote here by $x$ and $y$, respectively. The variables $u,v$ may then be viewed as the corresponding velocity components. The term $K$ is a thermodynamic forcing term, which comprises the heat sources driving the motion. We refer to \cite{CJ22} for precise information on the physical background and the derivation of the model. It is important to point out that, apart from the special oscillatory solution found in \cite[Section 6~(d)]{CJ22} and the travelling wave solutions investigated in \cite{CJ23}, no mathematical results related to~\eqref{PB} are available yet. This paper is thus a first attempt at laying the groundwork of the mathematical analysis of this model, which, in the future, may hopefully help to add additional insight into this fascinating meteorological phenomenon to that which is already available.

An important observation with respect to the evolution problem \eqref{PB} is that the 
equations~\eqref{EE}$_2$ and \eqref{BC}$_2$ may be used to eliminate the unknown $v$ from the problem,
 but only at the price of having to include some nonlocal terms that involve $u$ instead. 
 Indeed, given~${z\coloneqq(x,y)\in\0}$, the fundamental theorem of calculus yields
  \[
  v(z)=\int_0^yv_y(x,s)\ds=-\int_0^y u_x(x,s)\ds \eqqcolon -Tu_x(z),
  \]
where, given $u\in L_2(\0)$, we define $Tu:\0\to\R$ by
  \begin{equation}\label{OpT}
Tu(z)\coloneqq  \int_0^y u (x,s)\ds, \qquad z=(x,y)\in\0.
  \end{equation}
We also mention the related model  consisting of the viscous primitive equations of large scale ocean and atmosphere dynamics considered in \cite{CaTi07,GuMaRo01} (and the references therein) where, 
similarly as in our context, the last velocity component is expressed via the fundamental theorem of calculus in terms of the other velocity components.
 This observation enables us to rewrite \eqref{EE} in the following form:
 \begin{subequations}\label{PB'}
\begin{equation}\label{EE'}
u_t=\mu\Delta u -u u_x + u_yTu_x+\alpha u-\beta Tu_x +K\quad \text{in $\0,$ $t>0,$} 
\end{equation}
subject to the Dirichlet boundary condition
\begin{equation}\label{BC'}
 u=0\quad \text{on $\p\0$, \, $t>0$,} 
\end{equation}
and the initial condition
\begin{equation}\label{IC'}
u(0)=u_0.
\end{equation}
\end{subequations}

\subsection{Main results}\label{SS:01}

The first goal of this paper is to establish the existence and uniqueness of classical solutions to \eqref{PB'} for sufficiently regular initial data and under the assumption that 
the function $K$ is locally Lipschitz continuous with respect to time in $H^r_D(\0)$ for some small $r>0$, that is,
\begin{equation}\label{Kcla}
K\in {\rm C}^{1-}([0,\infty),H^{r}_D(\0)),
\end{equation}
see Theorem \ref{MT1} and Section~\ref{SS:11}.

\begin{thm}[Local well-posedness]\label{MT1}
Let $s\in (1,2)$, $r\in(0,s)\setminus\{1/2\}$, $\alpha\coloneqq s/2$, and assume that~${K = K(t,z)}$ satisfies \eqref{Kcla}. 
Then, for each $u_0\in H^{s}_D(\0)$, the Cauchy problem~\eqref{PB'} possesses a unique maximal classical solution
 \[u \in {\rm C}([0,T^+),H^{s}_D(\0))\cap {\rm C}((0,T^+),H^{2}_D(\0))\cap  {\rm C}^1((0,T^+),L_2(\0))\cap  {\rm C}^{\alpha}([0,T^+),L_2(\0)),\]
where $T^+ = T^+(u_0)>0$ is the maximal time of existence of the solution.
 Moreover, the map $(t,u_0) \mapsto u(t;u_0)$ is a semiflow  on $H^{s}_D(\0)$.
\end{thm}

In order to prove Theorem~\ref{MT1} we reformulate \eqref{PB'} in a suitable functional analytic framework as a quasilinear evolution problem and apply abstract theory for these types of problems from ~\cite{Am93} 
(see also~\cite{MW20}). 
Although \eqref{PB'}$_1$ has a semilinear structure,  the application of the quasilinear theory enables us to consider  more general initial data. 

Our  second main goal  is to show that, for each $u_0\in L_2(\0)$, and under the assumption  
\begin{equation}\label{Kwek}
K\in L_2(0,t; L_2(\0)) \qquad\text{for all $t>0$},
\end{equation}
the evolution problem~\eqref{PB'} possesses a global weak solution, as stated in Theorem~\ref{MT3}. 
Our approach relies on a two-step approximation strategy and  on an equivalent formulation of~\eqref{PB'} in terms of the new variable
\begin{equation}\label{subst}
w(t,z)\coloneqq u(t,z)e^{-\gamma t},\qquad t\geq 0,\, z\in\0,
\end{equation}
 see the system \eqref{PB''} below, where the positive constant $\gamma$ is defined as
\begin{equation}\label{gamma}
\gamma\coloneqq 1+\alpha+\frac{\beta^2}{2\mu}.
\end{equation}
The reason for this change of variables is motivated by the fact that 
 the $L_2$-functional
\[
\E(u)\coloneqq \frac{1}{2}\int_{\0}u^2\dz,
\] 
is not non-increasing along solutions the $L_2$-functional (if $K=0$), as we have
\begin{align*}
\frac{{\rm d}\E(u(t))}{{\rm d}t}&=\int_\0 u(t)\p_tu(t)\dz\\[1ex]
&=\int_\0 u(t) \big[\mu\Delta u -u u_x + u_yTu_x+\alpha u-\beta Tu_x +K\big](t)\dz\\[1ex]
&=\int_\0 \big[-\mu|\nabla u| ^2 + \alpha u^2-\beta uTu_x+Ku\big](t)\dz
\end{align*}
for $t>0$. 
However, it is important to note that  in the energy balance  the nonlinear terms vanish.
This motivates the substitution \eqref{subst}, which is a trick which finds application for example in the proof of the weak parabolic maximum principle.
With $\gamma$ chosen as in \eqref{gamma}, it is easy to see that the functional $\E$ is non-increasing along $w$ (if $K=0$), since
\begin{align}\label{Eneqw}
\frac{{\rm d}\E(w(t))}{{\rm d}t}+\frac{1}{2}\int_\0 \big[\mu|\nabla w| ^2 + w^2\big](t)\dz\leq\frac{1}{2}\int_\0 K^2(t)\dz,\qquad t>0,
\end{align}
see Section~\ref{Sec:2}.
The estimate \eqref{Eneqw} is essential for the proof of Theorem~\ref{MT3}.

\begin{thm}\label{MT3}
Given   $u_0\in L_2(\0)$ and $K=K(t,z)$ that satisfies \eqref{Kwek}, there exists a global weak solution $u$  to~\eqref{PB'} with the following properties:
 \begin{itemize} 
 \item[(i)] for all $t>0$ we have $u\in  L_\infty(0,t;L_2(\0))\cap  L_2(0,t;H^1(\0));$
 \item[(ii)] for all $t>0$ and all $\phi\in{\rm C}^\infty([0,t]\times\0) $ with the property  that there exists $M>0$ with  ${\rm  \supp\,}\phi(\tau)\subset (-2^M,2^M)\times (0,1) $ for all $\tau\in[0,t],$ we have
   \begin{equation*}
 \begin{aligned}
 &\hspace{-1cm}\int_{\0} u(t) \phi(t)\dz-\int_{\0} u_0 \phi(0)\dz-\int_0^t\int_{\0} u \p_t \phi\dz\dtau\\[1ex]
 &+\int_0^t\int_{\0}  \mu\nabla u\cdot\nabla \phi +u T\p_x u\p_y\phi+ \phi\big(2u\p_xu-\alpha u+\beta T\p_x u- K\big) \dz\dtau=0;
 \end{aligned}
 \end{equation*}
  \item[(iii)]   for almost all $t>0$  
\begin{equation}\label{EnE}
\E(u(t)e^{-\gamma t})+\frac{1}{2}\int_0^t\int_{\0}e^{-2\gamma \tau}\big(\mu |\nabla u|^2+u^2\big)\dz\dtau\leq  \E(u_0) +  \frac{1}{2}\int_0^t\int_{\0}K^2\dz\dtau. 
\end{equation}
 \end{itemize}
\end{thm}
The proof of this result is presented in Section~\ref{Sec:2}.
After reformulating~\eqref{PB'} in terms of the new variable~$w$, see~\eqref{PB''}, we establish~\eqref{Eneqw} and then consider the problem
 \eqref{PB''} on the rectangle~${\0^N\coloneqq(-2^N,2^N)\times(0,1)}$ with $N\geq 1$,
  a regularized $K$, and homogeneous Dirichlet boundary conditions on $\p\0^N$, see \eqref{PBN}.
  Using a Galerkin scheme,  we prove  in Section~\ref{SS:21},  that the problem \eqref{PBN} possesses a weak solution $w^N$, see  Proposition~\ref{P:1}.
In a second step we prove that the sequence of weak solutions $(w^N)$ converges, in a suitable sense, towards a weak solution to \eqref{PB''}, see Theorem~\ref{MT2}.
Then, in  a final step, we come back to the original unknown $u$ and deduce from Theorem~\ref{MT2} that $u$ is a weak solution to \eqref{PB'}, as stated in Theorem~\ref{MT3}.

 \section{Local well-posedness of \eqref{PB'}}\label{Sec:1}
 The main goal of this section is to prove the local well-posedness result stated in Theorem~\ref{MT1}.
 To this end we first introduce the function spaces which we are going to use in our analysis and provide some useful estimates for the operator $T$ defined in \eqref{OpT}.

\subsection{Preliminaries}\label{SS:11} Let us fix some notation. 
Given a Banach space $E$, an interval~${I \subset \R}$, $n\in\N$, and $\alpha\in(0,1)$,  we denote by ${\rm C}^n(I,E)$ the space of all
$n$-times continuously differentiable functions in $I$ and ${\rm C}^{n+\alpha}(I,E)$ is its subspace   that contains
only functions with locally $\alpha$-H\"older continuous $n$-th derivative.
Moreover,  ${\rm C}^{1-}(I,E)$ is the space of all locally Lipschitz  continuous functions on $I$.
Similarly, given Banach spaces  $X$ and $Y$,   ${\rm C}^{1-}(X,Y)$ is the space of all  locally Lipschitz continuous functions from $X$ to $Y.$ 
Let further $\mathcal{L}(X,Y)$ denote the Banach space of all bounded linear maps from $X$ to $Y$ (if $X=Y$, we set $\mathcal{L}(X) = \mathcal{L}(X,X))$. 

Given an open subset $\mathcal{O}\subset\R^2$, we let $L_2(\mathcal{O})$ be the space of square integrable functions on $\mathcal{O}$ 
and $H^s(\mathcal{O})$ denotes the Bessel potential space of order  $ s\in (0,\infty)$.
We further define~$H^1_0(\mathcal{O})$ as the closure of the set ${\rm C}^\infty_0(\mathcal{O})$ of  smooth functions with compact support in $\mathcal{O}$ in $H^1(\mathcal{O})$.

For $s\in[0,2]$ we set 
\[
H^s_D(\0)=
\left\{
\arraycolsep=1.4pt
\begin{array}{lcl}
H^{s}(\0)\,,&\quad \text{if $s\in(0,1/2)$,} \\[1ex]
 \{u \in H^{s}(\0)\,:\, \text{$u=0$ on $\p\0$}\}\,, & \quad\text{if $s\in(1/2,2]$.}
\end{array}
\right.
\]
Letting $[\cdot,\cdot]_\theta$ be the complex interpolation functor with exponent $\theta\in[0,1]$, we infer from \cite[Theorem 13.3]{Am84} that 
\begin{equation}\label{interpr}
[L_2(\0),H^2_D(\0)]_\theta=H^{2\theta}_D(\0).
\end{equation}

Given Banach spaces with   $E_1$ and $E_0$ with dense embedding $E_1\hookrightarrow E_0$, we follow ~\cite{Am95} and set
\begin{equation*} 
\kH(E_1,E_0)=\{A\in\kL(E_1,E_0)\,:\, \text{$-A$ generates an analytic semigroup in $\kL(E_0)$}\}.
\end{equation*}

\subsection{The abstract formulation} 
In order to prove Theorem \ref{MT1}, we shall establish some lemmas as preparation. 
The first one provides an important mapping property of the operator~$T$.
\begin{lemma}\label{operator_T}
Given $s\in[0,2]$, it holds that $T\in\kL(H^s(\0))$.
\end{lemma}
\begin{proof}
We first verify the claim for $s=0$. 
To this end we first take  $u\in {\rm C}^\infty_0(\0)$ and compute, in view of H\"older's inequality and Fubini's theorem, that 
\begin{equation}\label{E12}
\|T u\|_2=\Big(\int_\0\Big|\int_0^yu(x,s)\ds\Big|^2\dz\Big)^{1/2}\leq \Big(\int_\0 \int_0^1 u^2(x,s)\ds \dz\Big)^{1/2}=\|u\|_2.
\end{equation}
Since $  {\rm C}^\infty_0(\0)$ is dense in $L_2(\0)$, this proves the claim for $s=0$.

In order to demonstrate the claim for $s=1$ we use the fact that ${\rm C}^\infty(\ov\0) \cap H^1(\0)$ is dense in~$H^1(\0)$, the estimate \eqref{E12}, and the observation that, 
given $u\in {\rm C}^\infty(\ov\0) \cap H^1(\0)$, we have 
\[
(Tu)_x=Tu_x\qquad \text{and}\qquad (Tu)_y=u,
\]
to obtain
 \[\Vert Tu\Vert_{H^1(\0)}^2 \leq 2\Vert w\Vert_\lt^2 + \Vert w_x\Vert_\lt^2 \leq 2\Vert w\Vert_{H^1(\0)}^2.\]
The claim for $s=2$ is now easily inferred from the result for $s=1$. 
In view of the interpolation property 
\[
[L_2(\0),H^2(\0)]_\theta=H^{2\theta}(\0),\qquad \theta\in[0,1]
\]
(see e.g. \cite[Corollary 11.4]{Am84}), the desired claim is a consequence of the result for $s\in\{0,2\}$.
\end{proof}

Let $E_0\coloneqq L_2(\0)$, $E_1\coloneqq H^2_D(\0)$, and set $E_\theta\coloneqq[E_0,E_1]_\theta$ for $\theta\in(0,1)$.
We can now reformulate \eqref{PB'} as the following evolution problem:
\begin{equation}\label{AFor}  
 \frac{{\rm d}u(t)}{{\rm d}t} =A(u(t))u(t) + K(t), \quad t>0, \qquad  u(0) = u_0,
 \end{equation}
 where $K\in {\rm C}^{1-}([0,\infty), E_{r/2})$ -- see our assumption~\eqref{Kcla}, as well as \eqref{interpr} -- while the operator~${A: E_{s/2}\to\kL(E_1,E_0)}$, $s\in(1,2),$ is defined by
 \[
A(u)v\coloneqq\mu\Delta v+ \alpha v - uv_x + u_y Tv_x - \beta Tv_x,\qquad u\in E_{s/2},\,  v\in E_1.
 \]
The next lemma shows that $A$ is well-defined and smooth with respect to its arguments.
\begin{lemma}\label{L:c1} Given $s\in(1,2)$, we have: 
\begin{itemize}
\item[(i)]$[(u,v)\mapsto u v]: H^{s}(\0)\times H^{s-1}(\0)\to L_2(\0)$ is continuous;
\item[(ii)] $[(u,v)\mapsto uTv]: H^{s-1}(\0)\times H^{(3-s)/2}(\0)\to L_2(\0)$ is continuous.
\end{itemize}
Moreover, $A\in {\rm C}^\infty(E_{s/2}, \kL(E_1,E_0))$.
\end{lemma}
\begin{proof}
In view of Lemma~\ref{operator_T} we have $T\in\kL(H^{3/2-s}(\0))$.
Therefore it suffices to show that the bilinear operator $[(u,v)\mapsto uv]$ 
is continuous from $H^{s}(\0)\times H^{s-1}(\0)$, respectively from~${H^{s-1}(\0)\times H^{3/2-s}(\0)}$, to $L_2(\0)$.
These properties are straightforward consequences of the multiplication result stated in \cite[Theorem 4.1]{Am91}.
This proves (i) and (ii). 
Finally, the property~${A\in {\rm C}^\infty(E_{s/2}, \kL(E_1,E_0))}$ is a direct consequence of these results and \eqref{interpr}.
\end{proof}

In order to apply the theory from \cite{Am93} to the setting of \eqref{AFor}, we are still bound to show that $-A(u)$ generates an analytic semigroup in $\kL(E_0) $ for all $u\in E_{s/2}$ and $s\in(1,2)$.
\begin{lemma}\label{L:gerp} 
Given  $s\in(1,2)$ and $u\in E_{s/2}$, it holds that  $-A(u)\in\kH(E_1,E_0)$. 
\end{lemma}
\begin{proof}
According to Lemma~\ref{operator_T} and Lemma~\ref{L:c1}, we have 
\begin{align*}
&[v\mapsto u_yTv_x]\in\kL(E_{(5-s)/4}, E_0),\\[1ex]
& [v\mapsto uv_x]\in\kL(E_{s/2}, E_0),\\[1ex]
&  [v\mapsto \alpha v-\beta Tv_x]\in\kL(E_{1/2}, E_0).
\end{align*}
The generator property $-\Delta\in\kH(E_1,E_0) $ of the Laplacian is established in \cite[Theorem~13.4]{Am91}.
Moreover, the operators considered above can be treated as  perturbations, and the desired generator result follows via \cite[Theorem~I.1.3.1~(ii)]{Am95}.
\end{proof}

\subsection{The proof of Theorem \ref{MT1}}\label{Sec:13}
Let $\wt  K$ denote the even reflection of $K$. Then, according to \eqref{Kcla}, we have $\wt  K\in {\rm C}^{1-}(\R,E_{r/2})$. 
Setting~${U\coloneqq(t,u),}$ the problem \eqref{AFor} is equivalent to the autonomous quasilinear evolution problem
\begin{equation}\label{AForG}
 \frac{{\rm d}U(t)}{{\rm d}t} =B(U(t))U(t) + f(U(t)), \quad t>0, \qquad  U(0) = (0,u_0),
\end{equation}
where
\[
B(U)=\begin{pmatrix}
0&0\\[1ex]
0&A(u)
\end{pmatrix}\qquad\text{and}\qquad 
f(U)\coloneqq
\begin{pmatrix}
1\\[1ex]
\wt K(t)
\end{pmatrix}.
\]
Let $s\in(1,2)$, $r\in(0,s)\setminus\{1/2\} $, and choose $\ov s\in(1,s)$.
Moreover, we define $F_i\coloneqq\R\times E_i $ for~${i\in\{0,\,1\}}$ and we set
\[
F_\theta\coloneqq[F_0,F_1]_\theta=\R\times E_\theta,\qquad\theta\in(0,1),
\]
see \cite[Proposition~I.2.3.3]{Am95}. Our assumption~\eqref{Kcla}, Lemma~\ref{L:c1},  Lemma~\ref{L:gerp}, and \cite[Corollary~I.1.6.3]{Am95}, imply that
\[
-B\in{\rm C}^\infty(F_{\ov s/2},\kH(F_1,F_0)) \qquad\text{and}\qquad f\in{\rm C}^{1-}(F_{\ov s/2},F_{r/2}). 
\]
Hence, we are in a position to apply the quasilinear parabolic theory from \cite{Am95} (see \cite[Theorem~12.1]{Am95} or \cite[Theorem 1.1]{MW20}) to \eqref{AForG}
 and establish in this way the claims of Theorem~\ref{MT1}.

 \section{Existence of  global weak solutions to \eqref{PB'}}\label{Sec:2}

In terms of the variable $w$ defined in \eqref{subst}, the problem  \eqref{PB'} may be reformulated as the system 
\begin{equation}\label{PB''}
\left.
\begin{array}{rrrr}
w_t=\mu\Delta w -e^{\gamma t}w w_x +e^{\gamma t} w_yTw_x+(\alpha-\gamma) w-\beta Tw_x+e^{-\gamma t}K \quad \text{in $\0,$ $t>0,$} \\[1ex]
 w=0\quad \text{on $\p\0$, \, $t>0$,}\\[1ex]
 w(0)=u_0.
\end{array}
\right\}
\end{equation}
We note that  the positive constant $\gamma$ is defined in \eqref{gamma} in such a way that the $L_2$-energy functional is non-increasing along (classical) solutions to \eqref{PB''} when $K=0$.
Indeed, for~${t>0,}$ we have
\begin{align*}
\frac{{\rm d}\E(w(t))}{{\rm d}t}&=\int_\0 w(t)\p_tw(t)\dz\\[1ex]
&=\int_\0 w(t) \big[\mu\Delta w -e^{\gamma t}w w_x +e^{\gamma t} w_yTw_x+(\alpha-\gamma) w-\beta Tw_x+e^{-\gamma t}K\big](t)\dz\\[1ex]
&=\int_\0 \big[-\mu|\nabla w| ^2 + (\alpha-\gamma) w^2-\beta wTw_x+e^{-\gamma t}Kw\big](t)\dz.
\end{align*}
In view of Lemma~\ref{operator_T}, H\"older's inequality, and Young's inequality  we estimate
\[
\|wTw_x\|_1\leq \|w\|_2\|Tw_x\|_2\leq\|w\|_2\|\nabla w\|_2 \leq \frac{|\beta|\|w\|_2^2}{2\mu}+\frac{\mu\|\nabla w\|_2^2}{2|\beta|},
\]
and 
\[
\|e^{-\gamma t}Kw\|_1 \leq \frac{\|w\|_2^2}{2} +\frac{\|K\|_2^2}{2},
\]
and therewith we conclude that 
\begin{align*}
\frac{{\rm d}\E(w(t))}{{\rm d}t}+\frac{1}{2}\int_\0 \big[\mu|\nabla w| ^2 + w^2\big](t)\dz\leq\frac{1}{2}\int_\0 K^2(t)\dz.
\end{align*}

This energy inequality is at the base of our construction of weak solutions to \eqref{PB''}, as stated below.
\begin{thm}\label{MT2}
Given   $u_0\in L_2(\0)$ and $K=K(t,z)$ that satisfies \eqref{Kwek}, there exists a global weak solution $w$  to~\eqref{PB''} with the following properties:
 \begin{itemize} 
 \item[(i)] for all $t>0$ we have $w\in  L_\infty(0,t;L_2(\0))\cap  L_2(0,t;H^1(\0));$
 \item[(ii)] for all $t>0$ and all $\xi\in {\rm C}^\infty_0(\0)$ it holds that
 \begin{equation}\label{Eq0}
 \begin{aligned}
&\int_{\0} w(t)\xi\dz-\int_{\0} u_{0} \xi\dz \\[1ex]
&\qquad+\int_0^t\int_{\0}\big[\mu\nabla  w \cdot\nabla \xi+e^{ \gamma \tau}\p_y\xi w T\p_xw  +2e^{ \gamma \tau}\xi w\p_xw\\[1ex]
&\hspace{2.4cm}-(\alpha-\gamma) w\xi+\beta \xi T\p_xw-e^{-\gamma \tau}K\xi\big] \dz\dtau=0 ; \\[1ex]
 \end{aligned}
 \end{equation}
  \item[(iii)] for almost all $t>0$  
\begin{equation}\label{EnEs}
\E(w(t))+\frac{1}{2}\int_0^t\int_{\0}\mu |\nabla w|^2+w^2\dz\dtau\leq  \E(u_0) +  \frac{1}{2}\int_0^t\int_{\0}K^2\dz\dtau. 
\end{equation}
 \end{itemize}
\end{thm}

In order to prove Theorem~\ref{MT2}, in the sequel we fix~${u_0\in L_2(\0)}$ and $K=K(t,z)$, which satisfies \eqref{Kwek},
and we  set 
\begin{equation}\label{kt}
k(t)\coloneqq\|K\|_{L_2(0,t+1; L_2(\0))}^2,\qquad t\geq0.
\end{equation}
For later purposes we need to regularize $K$.
To this end we set  
\[
\wt K(t,z)\coloneqq
\left\{
\begin{array}{clll}
K(t,z)&,& (t,z)\in[0,\infty)\times\0, \\[1ex]
0&,&(t, z)\in\R^3\setminus\big([0,\infty)\times\0\big).
\end{array}
\right.
\]
Then $\wt K\in L_2((-t,t)\times\R^2)$ for all $t>0$.
Given $1\leq n\in \N$ we further set 
\begin{equation}\label{Kn}
K_n\coloneqq\wt K*\rho_{1/n},
\end{equation} where $(\rho_\e)_{\e>0} $ is a standard mollifier in $\R^3$.
Then $K_n\in{\rm C}^\infty(\R^3)$ satisfies
\begin{equation}\label{Pro1}
K_n\to K \qquad\text{in $L_2(0,t, L_2(\0))$ for all $t>0$} 
\end{equation}
and 
\begin{equation}\label{Pro2}
\|K_n\|_{L_2(0,t, L_2(\0))}^2\leq k(t)\qquad\text{for all $t>0$ and $1\leq n\in\N$}. 
\end{equation}

In a first step we consider  the problem \eqref{PB''} on the rectangle~${\0^N\coloneqq(-2^N,2^N)\times(0,1)}$ with $N\geq 1$ (see \eqref{PBN})
 and construct  via a Galerkin scheme a weak solution $w^N$  to \eqref{PBN} (see Proposition~\ref{P:1} below).
Then, in Section~\ref{SS:22}, we prove that the sequence   $(w^N)_N$ of weak solutions to \eqref{PBN} converges, in a suitable sense, towards a weak solution to \eqref{PB''}.

\subsection{A Galerkin scheme in bounded domains}\label{SS:21} 

In this section we consider the problem
\begin{equation}\label{PBN}
\left.
\begin{array}{rrrr}
w_t=\mu\Delta w -e^{\gamma t}w w_x +e^{\gamma t} w_yTw_x+(\alpha-\gamma) w-\beta Tw_x+e^{-\gamma t}K \quad \text{in $\0^N,$ $t>0,$} \\[1ex]
 w=0\quad \text{on $\p\0^N$, \, $t>0$,}\\[1ex]
 w(0)=u_0^N\coloneqq u_0|_{\0^N}, 
\end{array}
\right\}
\end{equation}
where $\0^N$ is  the rectangle 
\[
 \0^N\coloneqq(-2^N, 2^N)\times (0,1), \qquad 1\leq N\in\N,
\]
and we  use a Galerkin scheme to construct a global weak solution $w^N$ to \eqref{PBN}, as stated in Proposition~\ref{P:1}.
Before stating this result we introduce the energy functional $\E_N$ by the formula
\begin{equation}\label{E_N}
\E_N(u)\coloneqq\frac{1}{2}\int_{\0^N}u^2\dz.
\end{equation}

\begin{prop}\label{P:1} Given $u_0\in L_2(\0)$, $K=K(t,z)$ satisfying \eqref{Kwek},   and $1\leq N\in\N$, 
there exists a global weak solution $w^N$  to~\eqref{PBN} with the following properties:
 \begin{itemize} 
 \item[(i)] for all $t>0$ 
  \[
w^N\in  L_\infty(0,t;L_2(\0^N))\cap  L_2(0,t;H^1(\0^N)) \cap W^1_{4/3}(0,t; H_D^2(\0^N)'),
 \]
 where $H^{2}_D(\0^N)\coloneqq \{u \in H^{2}(\0^N)\,:\, u=0 \,\text{ on $\p\0^N$}\}$;
 \item[(ii)] for all $t>0$ and all $\xi\in H^2_D(\0^N)$ we have
 \begin{equation}\label{Equ1}
 \begin{aligned}
 \begin{aligned}
\int_{\0^N} w^N(t)\xi\dz&=\int_{\0^N} u_{0}^N \xi\dz \\[1ex]
&\quad +\int_0^t\int_{\0^N}\big[-\mu\nabla  w^N \cdot\nabla \xi- e^{ \gamma \tau}\p_y\xi w^N T\p_xw^N  -2e^{ \gamma \tau}\xi w^N\p_xw^N\\[1ex]
 &\hspace{2.35cm}+(\alpha-\gamma) w^N\xi-\beta \xi T\p_xw^N+e^{-\gamma \tau}K\xi\big] \dz\dtau;
 \end{aligned}
 \end{aligned}
 \end{equation}
  \item[(iii)] for almost all $t>0$ we have
  \begin{equation}\label{EnEs0}
 \E_N(w^N(t))+\frac{1}{2}\int_0^t\int_{\0^N}\mu |\nabla w^N|^2+|w^N|^2\dz\dtau\leq\E(u_0) + \frac{1}{2}\int_0^t\int_{\0^N}K^2\dz\dtau.
 \end{equation}

 \end{itemize}
 \end{prop}

 The proof of Proposition~\ref{P:1} is postponed to the end of this section, as it requires some preparation.
We first collect in Lemma~\ref{L:3} 
some classical results related to the spectrum of the Laplace operator with homogeneous Dirichlet boundary conditions.

\begin{lemma}\label{L:3} Given $1\leq N\in\N$, let $-\lambda_k^N$, $1\leq k\in\N$, with
\[
0<\lambda_1^N\leq \lambda_2^N\leq \ldots
\] 
denote the eigenvalues of the Dirichlet Laplacian  and let $\phi^N_k$ denote an  eigenfunction corresponding to the eigenvalue $\lambda_k^N$, $k\geq 1$. 
Then $\lambda_k^N\to\infty $ for $k\to\infty$ and the eigenfunctions can be chosen such that the set $\{\phi^N_k\, :\, k\geq 1\} $ is an orthonormal basis of $L_2(\0).$
Furthermore, letting  
\[
u_n\coloneqq\sum_{k=1}^n\langle u|\phi^N_k\rangle_2 \phi^N_k,
\]
where $\langle \cdot|\cdot\rangle_2 $ is the $L_2(\0^N)$-scalar product, it holds that 
\begin{equation}\label{Eq2}
u_n\underset{n\to\infty}\to u \qquad \text{in $H^2(\0^N)$}
\end{equation}
provided that $u\in H^{2}_D(\0^N)$.
\end{lemma}
\begin{proof}
It is well-known that $u_n\to u$ in $L_2(\0^N)$.
Hence, in view of  $\Delta u\in L_2(\0^N)$, we also have that 
\[
(\Delta u)_n =\sum_{k=1}^n\langle \Delta u|\phi^N_k\rangle_2 \phi^N_k\underset{n\to\infty}\to\Delta u \qquad\text{in $L_2(\0^N)$}.
\] 
 Since $u\in H^{2}_D(\0^N),$ Stokes' theorem yields
\[
\langle \Delta u|\phi^N_k\rangle_2=\langle u| \Delta\phi^N_k\rangle_2=-\lambda_k^N\langle u| \phi^N_k\rangle_2, \quad k\geq1,
\]
which shows that $(\Delta u)_n=\Delta  u_n .$
Hence $(1-\Delta) u_n \to(1-\Delta) u$ for $n\to\infty$ in $L_2(\0^N). $ 
Moreover, since $\0^N$ is convex,  $1-\Delta:H^2_D(\0^N)\to L_2(\0^N)$ is an isomorphism (see e.g. \cite[Regularity Theorem 7.2]{Br07}) and we may conclude that \eqref{Eq2} holds true.
\end{proof}

 The operator $T$ possesses similar properties as in Lemma~\ref{operator_T} when acting upon functions defined on $\0^N$.
\begin{lemma}\label{L:1'}
Given $N\geq 1,$ it holds that   $\|T\|_{\kL(L_2(\0^N))}\leq 1$.
\end{lemma}
\begin{proof}
 The claim follows by arguing as in Lemma~\ref{operator_T}.
\end{proof}

In the following $1\leq N\in\N$ is fixed. Given $1\leq n\in\N$, we  set $V_n\coloneqq{\rm span}\{\phi_1^N,\ldots,\phi_n^N\}$ and we define
\[
u_{0,n}\coloneqq\sum_{k=1}^n\omega_k \phi^N_k\in V_n, \qquad \text{where $\omega_k\coloneqq\langle u_0|\phi^N_k\rangle_2$,\, $1\leq k\leq n$.}
\]
We then have
\begin{equation}\label{estu0}
\|u_{0,n}\|_{L_2(\0^N)}\leq \|u_0\|_{L_2(\0^N)}\leq \|u_0\|_{L_2(\0)},\qquad n\geq1,
\end{equation}
and, moreover,
\begin{equation}\label{conv0}
u_{0,n}\to u_0^N\qquad\text{in $L_2(\0^N)$}.
\end{equation}

We next look for a solution $w_n\coloneqq w_n^N$ to \eqref{PBN} of the form
\[
w_n(t,z)=\sum_{k=1}^n F_k^N(t)\phi^N_k(z),\qquad t\geq 0,\, z\in\0^N,
\] 
such that $(F_1^N,\ldots, F_n^N):[0,\infty)\to\R^n$ is a continuously differentiable function with
\begin{equation}\label{icn}
(F_1^N,\ldots, F_n^N)(0)= \omega \coloneqq (\omega_1,\ldots,\omega_n)
\end{equation}
and such that  $w_n$ satisfies  the equation \eqref{PBN}$_1$ with $K$ replaced by $K_n$, see \eqref{Kn}, at each time~$t\geq0$ in a weak sense,
that is, when testing the equation  with functions from $V_n$.
We note that \eqref{icn} is equivalent to the equation $w_n(0)=u_{0,n}$.
In order to establish the existence of the solution $w_n$, we test \eqref{PBN}$_1$ at $t\geq 0$ with $\phi_\ell^N$, $1\leq \ell\leq N$, and obtain that 
\begin{align*}
\frac{{\rm d} F_\ell^N(t)}{{\rm d}t}&=(\alpha-\gamma-\mu \lambda_\ell^N)F_\ell^N(t)\\[1ex]
&\quad+\int_{\0^N}\phi_\ell^N(z)\bigg[-e^{\gamma t}\Big(\sum_{k=1}^n F_k^N(t)\phi^N_k(z)\Big)\Big(\sum_{k=1}^n F_k^N(t)\p_x\phi^N_k(z)\Big)  \\[1ex]
&\hspace{2.8cm}+e^{\gamma t}\Big(\sum_{k=1}^n F_k^N(t)\p_y\phi^N_k(z)\Big)\Big(\sum_{k=1}^n F_k^N(t)(T\p_x\phi^N_k)(z)\Big) \\[1ex]
&\hspace{2.8cm}-\beta\sum_{k=1}^n F_k^N(t)(T\p_x \phi^N_k)(z)+e^{-\gamma t}K_n(t,z)\sum_{k=1}^n F_k^N(t)\phi^N_k(z)\bigg]\dz  .
\end{align*}
Hence $\xi\coloneqq(F_1^N,\ldots, F_n^N)$ solves an initial value problem  for an ordinary differential equation of the form
\begin{equation} \label{ode}
\xi'=\Psi^n(t,\xi),\qquad \xi(0)=\omega,
\end{equation}
where, since $K_n$ is smooth, it holds that $\Psi^n\in {\rm C}^\infty(\R^{n+1},\R^n) $.
By standard  theory (see e.g.~\cite{Am90,PW10}), we conclude that there exists  a maximal solution $ (F_1^N,\ldots, F_n^N)\in{\rm C}^1([0,T^+_n),\R^n)$
to~\eqref{ode} with the maximal existence time $T^+_n\in(0,\infty]$.

We next prove that the solution  is bounded on each   interval $[0,T]\cap [0,T^+_n)$, with $T>0$, which ensures that $T^+_n=\infty$.
Testing the equation \eqref{PBN}$_1$ with $w_n(t)\in V_n$, $t\in [0,T^+_n)$, we obtain that 
 \begin{equation*}
 \begin{aligned}
 \frac{{\rm d} \E_N(w_n(t))}{{\rm d}t}&=\int_{\0^N}w_n(t)\p_tw_n(t)\dz\\[1ex]
 &=\int_{\0^N} \big[-\mu|\nabla w_n| ^2 + (\alpha-\gamma) w_n^2-\beta w_nT\p_x w_n+e^{-\gamma t}K_nw_n\big](t)\dz.
 \end{aligned}
 \end{equation*}
 
Recalling Lemma~\ref{L:1'},  \eqref{gamma}, and the estimates \eqref{Pro2} and \eqref{estu0}, we  may argue as in the derivation of \eqref{Eneqw} to obtain -- after integration on~$[0,t]$ --
that 
 \begin{equation}\label{EnEs1}
 \begin{aligned}
 \E_N(w_n(t))+\frac{1}{2}\int_0^t\int_{\0^N}\mu |\nabla w_n|^2+|w_n|^2\dz\dtau&\leq\E(u_0)+\frac{1}{2}\int_0^t\int_{\0^N} K_n^2\dz\dtau  \\[1ex]
&\leq\E(u_0)+\frac{k(t)}{2} 
 \end{aligned}
 \end{equation}
 for all $t\in [0,T^+_n)$.
Since
\[
 \E_N(w_n(t))=\frac{1}{2}  \sum_{k=1}^n(F_k^N(t))^2,\qquad t\in [0,T^+_n),
\]
the estimate \eqref{EnEs1} implies that the solution $  (F_1^N,\ldots, F_n^N) $ is globally defined. 

In the remainder of this section we denote by $C$ positive constants which may depend only on  $\alpha,\, \beta,$ and~$\mu $.

\begin{lemma}[A priori estimates]\label{L:4}
There exists a positive constant $C$ such that  for all~$t\geq 0$ and~$n\geq 1$  we have
 \begin{equation}\label{estfun}
 \int_0^t\| w_n(\tau)\|_{H^1(\0^N)}^{2}\dtau\leq C\big(\E(u_0)+k(t)\big)
 \end{equation}
 and 
 \begin{equation}\label{estder}
 \int_0^t\|\p_t w_n(\tau)\|_{H^2_D(\0^N)'}^{4/3}\dtau\leq C\big(1+\E(u_0)+k(t)\big)^2e^{Ct} .
 \end{equation}
\end{lemma}
 \begin{proof}
 The estimate \eqref{estfun} follows immediately from \eqref{EnEs1}.
 
 It remains to prove \eqref{estder}. Let therefore $\xi\in  H^2_D(\0^N).$
 According to Lemma~\ref{L:3}, we may decompose $\xi$ as $\xi=\xi_n+\xi_0$, where~$\xi_n\in V_n$ and~${\xi_0\coloneqq\xi-\xi_n\in H^2_D(\0^N)}$  are orthogonal  in~$H^2(\0^N)$; thus
 $$\|\xi_n\|_{H^2}\leq \|\xi\|_{H^2}\qquad\text{for all $n\geq 1$.}$$
Hence, given $t\geq 0$, we have 
 \begin{align*}
&\hspace{-0.5cm} \frac{{\rm d}}{{\rm d}t}\int_{\0^N}w_n(t)\xi\dz=\int_{\0^N}\p_t w_n(t)\xi\dz=\int_{\0^N}\p_t w_n(t) \xi_n\dz\\[1ex]
 &=\int_{\0^N} \big[\mu\Delta w_n -e^{\gamma t}w_n \p_xw_n + e^{\gamma t}\p_yw_nT\p_xw_n\\[1ex]
 &\hspace{1.5cm}+(\alpha-\gamma) w_n-\beta T\p_xw_n+e^{-\gamma t}K_n\big](t) \xi_n\dz.
 \end{align*}
 Stokes' theorem, together with the observation that $\p_y Tu=u $ for $u\in H^1(\0^N)$, leads us to the identity
 \[
\int_{\0^N} \xi_n \p_yw_n(t)T\p_xw_n(t)\dz=-\int_{\0^N} \p_y\xi_n w_n(t) T\p_xw_n(t) +\xi_n w_n(t) \p_x w_n(t)\dz,
 \] 
 and therefore we have
 \begin{equation}\label{soliden} 
  \begin{aligned}
\int_{\0^N}\p_t w_n(t)\xi\dz &=\int_{\0^N}\big[ -\mu\nabla  w_n(t)\cdot\nabla \xi_n- e^{\gamma t}\p_y\xi_n w_n(t)T\p_xw_n(t)+(\alpha-\gamma) w_n(t)\xi_n\\[1ex]
 &\hspace{1,42cm}-2e^{\gamma t}\xi_n  w_n (t)\p_x   w_n (t) -\beta \xi_n T\p_xw_n(t)+e^{-\gamma t}K_n(t)\xi_n\big] \dz.
 \end{aligned}
 \end{equation}
The relation  \eqref{soliden}, combined with Lemma~\ref{L:1'}, now yields
 \begin{align*}
&\hspace{-0.5cm}  \Big|\int_{\0^N}\p_t w_n(t)\xi\dz\Big|\\[1ex]
  &\leq C\big(\|  w_n(t)\|_{H^1}\| \xi_n\|_{H^1}+ e^{\gamma t}\|\xi_n\|_{W^1_4} \|w_n(t)\|_4\|w_n(t)\|_{H^1}+\|K_n(t)\|_2\|\xi_n\|_2\big).
  \end{align*}
 The embedings $H^2(\0^N)\hookrightarrow W^1_4(\0^N)$, $H^{1/2}(\0^N)\hookrightarrow L_4(\0^N)$,  the estimate 
 \[
 \|u\|_{H^{1/2}}\leq C\|u\|_2^{1/2}\|u\|_{H^1}^{1/2},\qquad u\in H^1(\0^N),
 \]
 and the bounds \eqref{Pro2} and \eqref{EnEs1} imply that there exists a positive constant $C$ such that 
 \begin{align*}
  \Big|\int_{\0^N}\p_t u_n(t)\xi\dz\Big|&\leq C\Big[1+k(t)^{1/2}+\big(1+k(t)+\E(u_0)\big)^{1/2}e^{Ct}\|  w_n(t)\|_{H^1}^{3/2}\Big]\| \xi_n\|_{H^2}\\[1ex]
  &\leq  C\big(1+k(t)+\E(u_0)\big)^{1/2}e^{Ct}\Big[1+\|  w_n(t)\|_{H^1}^{3/2}\Big]\| \xi\|_{H^2}.
 \end{align*}
 Therefore 
 \[
 \|\p_t w_n(t)\|_{H^2_D(\0^N)'}\leq C\big(1+k(t)+\E(u_0)\big)^{1/2}e^{Ct}\Big[1+\|  u_n(t)\|_{H^1}^{3/2}\Big]\qquad\text{for all $t\geq 0$,}
 \]
 and the desired claim \eqref{estder} follows by applying \eqref{estfun}.
 \end{proof}
 
 The bounds provided in  Lemma~\ref{L:4} enable us to obtain some compactness results for the sequence $(w_n)_n$.
 Indeed, taking into account that
\[
H^1(\0^N)\hookrightarrow H^{1/2}(\0^N)\hookrightarrow H^2_D(\0^N)',
\] 
 with compact   embedding $H^1(\0^N)\hookrightarrow H^{1/2}(\0^N)$,  we infer from \eqref{estfun},  \eqref{estder},   and~\cite[Corollary~4]{Si87}
 that the sequence $(w_n)_n$ is relatively compact in $L_2(0,t; H^{1/2}(\0^N))$.
 Moreover, a standard Cantor's diagonal argument allows us to conclude that there exists a subsequence of~$(w_n)_n$ (not relabeled) and a function $w^N:[0,\infty)\times\0^N\to\R$ such that 
 \begin{align}\label{conv1}
 &w_n\to w^N&&\text{in $L_2(0,t;H^{1/2}(\0^N))$ for all $t>0$,}\\[1ex]
\label{conv2}
 &w_n(t)\to w^N(t)&&\text{in $L_2(\0^N)$ for almost all $t>0$,}\\[1ex]
\label{conv3}
 &\nabla w_n\rightharpoonup \nabla w^N&&\text{in $L_2((0,t)\times \0^N)$ for all $t>0$.}
\end{align} 

Recalling \eqref{estder}, we also have
 \begin{equation}\label{conv4}
 \p_t w_n\rightharpoonup \p_tw^N\qquad\text{in $L_{4/3}(0,t; H_D^2(\0^N)')$ for all $t>0$.}
\end{equation}

 The convergences \eqref{conv1}--\eqref{conv4} enable us to spell out the proof of Proposition~\ref{P:1}.

 \begin{proof}[Proof of Proposition~\ref{P:1}]
 We show that the function $w^N$ identified above by the convergences~\eqref{conv1}--\eqref{conv4} satisfies all the  properties stated in Proposition~\ref{P:1}.
 
 The property (i) follows from \eqref{conv1}--\eqref{conv4}  by  passing to $\liminf$ in~\eqref{EnEs1} and \eqref{estder}.
 
  Moreover, passing to $\liminf$ in \eqref{EnEs1}, we deduce from \eqref{Pro1} and \eqref{conv1}--\eqref{conv3} that the energy estimate~\eqref{EnEs0} is valid.
  
  It remains to verify (ii). To this end we integrate over $[0,t]$, for given $\xi\in H^2_D(\0^N)$, the identity~\eqref{soliden}, to deduce that  for all $n\geq 1$ we have
 \begin{equation}\label{jjj} 
  \begin{aligned}
\int_{\0^N} w_n(t)\xi_n\dz&=\int_{\0^N} u_{0,n} \xi_n\dz \\[1ex]
&\quad +\int_0^t\int_{\0^N}\big[ -\mu\nabla  w_n \cdot\nabla \xi_n- e^{\gamma \tau}\p_y\xi_n w_n T\p_xw_n  -2e^{\gamma \tau}\xi_n w_n\p_xw_n\\[1ex]
 &\hspace{2.35cm}+(\alpha-\gamma) w_n\xi_n-\beta \xi_n T\p_xw_n+e^{-\gamma \tau}K_n\xi_n\big] \dz\dtau.
 \end{aligned}
 \end{equation}
 In view of  \eqref{Eq2}, \eqref{conv0}, and \eqref{conv2} we have for almost all $t>0$ that 
 \[
 \int_{\0^N} w_n(t)\xi_n\dz\to  \int_{\0^N} w^N(t)\xi\dz \qquad\text{and}\qquad \int_{\0^N} u_{0,n} \xi_n\dz\to  \int_{\0^N} u_0^N \xi\dz.
 \]
Besides, since $\xi_n\to \xi$ in $H^1(\0^N)$, the  convergences~\eqref{Pro1},~\eqref{conv1}, and~\eqref{conv3}  ensure that 
 \begin{align*}
& \int_0^t\int_{\0^N}-\mu\nabla  w_n \cdot\nabla \xi_n+(\alpha-\gamma) w_n \xi_n + e^{-\gamma \tau}K_n\xi_n\dz\dtau\\[1ex]
 &\qquad\to \int_0^t\int_{\0^N}-\mu\nabla  w^N \cdot\nabla \xi+(\alpha-\gamma) w^N \xi+e^{-\gamma \tau}K\xi \dz\dtau\qquad\text{for all $t>0$.}
 \end{align*}

We next observe that, since $H^{1/2}(\0^N)\hookrightarrow L_4(\0^N)$, the relations \eqref{Eq2} and \eqref{conv1} lead  to~${\xi_n\to\xi}$ in $L_\infty(0,t; L_4(\0^N))$ and $w_n\to w^N$ in $L_2(0,t; L_4(\0^N))$, hence
\begin{equation}\label{uxi}
w_n\xi_n\to w^N\xi\qquad\text{in $L_2((0,t)\times\0^N)$ for all $t>0$.}
\end{equation}
Combining the  convergences \eqref{conv3} and \eqref{uxi}, we arrive at
\[
\int_0^t\int_{\0^N}e^{\gamma \tau}\xi_n w_n\p_xw_n \dz\dtau\to \int_0^t\int_{\0^N}e^{\gamma \tau}\xi w^N\p_xw^N \dz\dtau.
\]

 It remains to pass to the limit in the terms that involve the nonlocal operator~$T$.
 Thus we define  the functions
 \[
 \zeta_n(t,x,s)\coloneqq \int_s^1( w_n(t)\p_y\xi_n +\beta \xi_n )(x,y)\,{\rm d}y \quad \text{and} \quad \zeta(t,x,s)\coloneqq \int_s^1( w^N(t)\p_y\xi +\beta \xi )(x,y)\,{\rm d}y
 \]
 for $t\geq 0$ and $(x,s)\in\0^N$, and note  that, in view of Fubini's theorem,  we have
 \begin{align}\label{vrr}
\int_0^t\int_{\0^N}e^{\gamma \tau}(w_n \p_y\xi_n  +\beta \xi_n )T\p_xw_n\dz\dtau =\int_0^t\int_{\0^N}e^{\gamma \tau}\p_xw_n(\tau,x,s)\zeta_n(\tau,x,s){\,\rm d}(x,s)\dtau.
 \end{align}
Using H\"older's inequality, we then find
\begin{align*}
\|\zeta_n(t)-\zeta(t)\|^2_{L_2(\0^N)}&\leq \int_{\0^N}\int_0^1|w_n(t)\p_y\xi_n  +\beta \xi_n -w^N(t)\p_y\xi  -\beta \xi|^2(x,y)\,{\rm d}y{\,\rm d}(x,s)\\[1ex]
&\leq \|w_n(t) \p_y\xi_n  -w^N(t) \p_y\xi  \|^2_{L_2(\0^N)}+\beta^2\|  \xi_n -  \xi\|^2_{L_2(\0^N)}\\[1ex]
&\leq \|w_n(t)   -w^N(t)   \|^2_{L_4(\0^N)}\|\p_y\xi_n\|_{L_4(\0^N)}^2\\[1ex]
&\quad+\|w^N(t)\|_{L_4(\0^N)}^2\| \p_y\xi_n  - \p_y\xi  \|^2_{L_4(\0^N)}\\[1ex]
&\quad+\beta^2\|  \xi_n -  \xi\|^2_{L_2(\0^N)} 
\end{align*}
for $t\geq0$.
Since $H^2(\0^N)\hookrightarrow W^1_4(\0^N)$, the relation~\eqref{Eq2} guarantees that   
$\p_y\xi_n\to\p_y\xi$ in~${L_\infty(0,t; L_4(\0^N))}$, and with $w_n\to w^N$ in $L_2(0,t; L_4(\0^N))$ we may conclude that 
\[
\zeta_n\to\zeta\qquad\text{in $L_2((0,t)\times\0^N)$ for all $t>0$.}
\]
The latter property, together with \eqref{conv3}, enables us to pass to the limit $n\to\infty$ in \eqref{vrr} to arrive    at
\begin{align*}
\int_0^t\int_{\0^N}e^{\gamma \tau}(w_n \p_y\xi_n  +\beta \xi_n )T\p_xw_n\dz\dtau\to \int_0^t\int_{\0^N}e^{ \gamma \tau}(w^N \p_y\xi  +\beta \xi )T\p_xw^N\dz\dtau.
 \end{align*}
 Finally, observing that both sides of \eqref{Equ1} are continuous with respect to $t$, we conclude that the identity~\eqref{Equ1} is satisfied for all $t> 0$, and the  proof is complete.
 \end{proof}

 \subsection{The proof of Theorem~\ref{MT2}}\label{SS:22}
 We now prove that the sequence $(w^N)_N$ provided by Proposition~\ref{P:1} is relatively compact in $L_2(0,t; H^s(\0^{M}))$ for all $s\in(0,1)$, $M\geq1$, and~$t>0$.
 This is one of the main ingredients which will make it possible to pass to the limit $N\to\infty$ in the relations \eqref{Equ1} and \eqref{EnEs0} and thus establish our main result, cf. Theorem~\ref{MT2}.
 Let therefore $M\geq 1$ be given.
Since $\0^M\subset\0^N$ for  $N\geq M$, we deduce from \eqref{EnEs0} that 
\begin{equation}\label{Bound1}
\text{$(w^N)_{N\geq M}$ is bounded in $L_\infty(0,t;L_2(\0^M))\cap L_2(0,t; H^1(\0^M))$ for all $t>0.$}
\end{equation}
We further note that the identity \eqref{Equ1} entails that 
 \begin{equation*}
 \begin{aligned}
&\hspace{-0.5cm}\langle\p_t  w^N(t),\xi\rangle_{H^2_D(\0^N)', H^2_D(\0^N)}\\[1ex]
&=\int_{\0^N}\big[-\mu\nabla  w^N(t) \cdot\nabla \xi- e^{\gamma t}\p_y\xi w^N(t) T\p_xw^N(t)  -2e^{ \gamma t}\xi w^N\p_xw^N\\[1ex]
 &\hspace{1.4cm}+(\alpha-\gamma) w^N(t)\xi-\beta \xi T\p_xw^N(t)+e^{-\gamma t}K(t)\xi\big] \dz
 \end{aligned}
 \end{equation*}
 for all $N\geq 1$, $\xi\in H^2_D(\0^N), $ and almost all $t> 0$. 
 Let $H^2_0(\0^M)$ denote the closure of the set~${\rm C}^\infty_0(\0^M)$ in $H^2(\0^M)$. 
 Given $N\geq M$, we clearly have~$H^2_0(\0^M)\hookrightarrow H^2_D(\0^N),$ and therefore~${\p_tw^N(t)\in H^2_D(\0^N)'\hookrightarrow H^2_0(\0^M)'}$, where, arguing as in the proof of Lemma~\ref{L:4}, it holds that  
 \begin{equation*} 
 \begin{aligned}
&\hspace{-0.5cm}\big|\langle\p_t   w^N(t),\xi\rangle_{H^2_0(\0^M)', H^2_0(\0^M)}\big|\\[1ex]
&=\bigg|\int_{\0^M}\big[-\mu\nabla  w^N(t) \cdot\nabla \xi- e^{\gamma t}\p_y\xi w^N(t) T\p_x w^N(t)  -2e^{ \gamma t}\xi w^N\p_x w^N\\[1ex]
 &\hspace{1.65cm}+(\alpha-\gamma) w^N(t)\xi-\beta \xi T\p_xw^N(t) +e^{-\gamma t}K(t)\xi\big] \dz\bigg|\\[1ex]
 &\leq C\bigg[1+\|K(t)\|_2+\bigg(1+\E(u_0)+\int_0^t\int_{\0}K^2\dz\dtau\bigg)^{1/2}e^{Ct}\|  w^N(t)\|_{H^1}^{3/2}\bigg]\|\xi\|_{H^2}.
 \end{aligned}
 \end{equation*}
 Consequently, for almost all $t>0$ we have that
 \[
 \|\p_t  w^N(t)\|_{H^2_0(\0^M)'}\leq C\bigg[1+\|K(t)\|_2+\bigg(1+\E(u_0)+\int_0^t\int_{\0}K^2\dz\dtau\bigg)^{1/2}e^{Ct}\|  w^N(t)\|_{H^1}^{3/2}\bigg],
 \]
 and with \eqref{Bound1} we deduce that 
 \begin{equation}\label{Bound2}
\text{$(\p_t w^N)_{N\geq M}$ is bounded in $L_{4/3}(0,t; H^2_0(\0^M)')$ for all $t>0.$}
\end{equation}

Due to the embeddings
\[
H^1(\0^M)\hookrightarrow H^{1/2}(\0^M)\hookrightarrow H^2_0(\0^M)'
\] 
 (where the first embedding is compact), the bounds \eqref{Bound1}--\eqref{Bound2},    and~\cite[Corollary~4]{Si87}
 we may conclude that the sequence $(w^N)_{N\geq M}$ is relatively compact in $L_2(0,t; H^{1/2}(\0^M))$.
 Moreover, a standard Cantor's diagonal argument allows us to conclude that there exists a subsequence of $( w^N)_N$ (not relabeled)  and a function $w:[0,\infty)\times\0\to\R$ such that 
 \begin{align}\label{cv1}
 &w^N\to w&&\text{in $L_2(0,t;H^{1/2}(\0^M))$ for all $t>0$, $M\geq1$,}\\[1ex]
\label{cv2}
 &w^N(t)\to w(t)&&\text{in $L_2(\0^M)$ for almost all $t>0$ and all  $M\geq1$,}\\[1ex]
\label{cv3}
 &\nabla w^N\rightharpoonup \nabla w&&\text{in $L_2((0,t)\times \0^M)$ for all $t>0$, $M\geq1$,}\\[1ex]
 \label{cv4}
 &\p_t w^N\rightharpoonup \p_t w&&\text{in $L_{4/3}(0,t; H^2_0(\0^M)')$ for all $t>0$, $M\geq1$.}
\end{align} 
We now infer from \eqref{EnEs0} that for $N\geq M$ and for almost all $t>0$ we have
\[
\int_{\0^M}(w^N(t))^2\dz+\int_0^t\int_{\0^M}\mu |\nabla w^N|^2+|w^N|^2\dz\dtau\leq 2\E(u_0) +  \int_0^t\int_{\0}K^2\dz\dtau. 
\]
Passing to $\liminf$ in the latter  estimate we find, in view of \eqref{cv1}--\eqref{cv3}, that 
\[
\int_{\0^M}(w(t))^2\dz+\int_0^t\int_{\0^M}\mu |\nabla w|^2+w^2\dz\dtau\leq 2\E(u_0) +  \int_0^t\int_{\0}K^2\dz\dtau. 
\]
  By the monotone convergence theorem, we may pass to the limit $M\to\infty$ in the latter relation and conclude that
\eqref{EnEs} is satisfied, which proves also the claim (i).

It remains to demonstrate the property \eqref{Eq0}. 
However, this follows by letting $N\to\infty$ in  \eqref{EnEs0} (with $\xi\in {\rm C}^\infty_0(\0)$), in view of the convergences \eqref{cv1}--\eqref{cv4}.
Since the details are similar to those in Proposition~\ref{P:1} we omit them here.

\subsection{The proof of Theorem~\ref{MT3}}\label{SS:24}

It is straightforward to deduce from the identity~\eqref{Eq0} that  for all $0\leq s<t$ and $\xi\in{\rm C}^\infty_0(\0)$ we have
 \begin{equation}\label{Eq1}
\int_{\0} w(t)\xi\dz-\int_{\0} w(s)\xi\dz+\int_s^t\int_{\0}\big[\mu\nabla  w \cdot\nabla \xi+J_1\p_y\xi +J_2\xi\big] \dz \dtau=0  ,
 \end{equation}
where
 \[
 J_1 \coloneqq e^{ \gamma t}  w T\p_xw\qquad\text{and}\qquad J_2\coloneqq 2e^{ \gamma t}  w\p_xw-(\alpha-\gamma) w+\beta T\p_xw-e^{-\gamma t}K.
 \]
 Let $t>0$ be given and choose  $\eta\in{\rm C}^\infty([0,t]\times\0) $ with the property  that there exists $M>0$ with  ${\rm  \supp\,}\eta(\tau)\subset \0^M $ for all $\tau\in[0,t],$
 where as before we set $\0^M\coloneqq (-2^{M},2^M)\times(0,1).$
Let further  $\Delta_n$, $1\leq n\in\N$, be the uniform partition $0=t_0<t_1<\ldots<t_n=t$  of the interval~${[0,t].}$

It then follows from \eqref{Eq1} that 
 \begin{equation}\label{1d}
 \begin{aligned}
 &\hspace{-1cm}\int_{\0^M} w(t) \eta(t)\dz-\int_{\0^M} u_0 \eta(0)\dz+\int_{\0^M} \sum_{i=1}^n\big(w(t_i)\eta(t_{i-1})-w(t_{i-1})\eta(t_{i})\big)\dz\\[1ex]
 &+\int_0^t\int_{\0^M}  \mu\nabla w\cdot\nabla \varphi_n +J_1\p_y\varphi_n+J_2\varphi_n \dz\dtau=0,
 \end{aligned}
 \end{equation}
 where
 \[
 \varphi_n(\tau,t)\coloneqq\sum_{i=1}^n(\eta(t_{i-1},z)+\eta(t_{i},z)){\bf 1}_{(t_{i-1},t_i]}(\tau),\qquad (\tau,z)\in[0,t]\times\0.
 \]
 
 We next pass to the limit $n\to\infty $ in all terms of \eqref{1d}.
To this end we first observe that 
 \begin{equation*}
 \left.
 \begin{array}{lll}
 \varphi_n \to 2\eta \\[1ex]
\nabla \varphi_n \to 2\nabla \eta 
 \end{array}
 \right\}\qquad\text{for $n\to\infty$ in $L_\infty((0,t)\times \0^M)$.}
 \end{equation*}
 Since $\nabla w, \, J_1,\, J_2\in L_1((0,t)\times \0^M)$ we obtain that 
 \begin{equation}\label{ref1}
 \int_0^t\int_{\0^M}  \mu\nabla w\cdot\nabla \varphi_n +J_1\p_y\varphi_n+J_2\varphi_n \dz\dtau\to  2\int_0^t\int_{\0}  \mu\nabla w\cdot\nabla \eta +J_1\p_y\eta+J_2\eta \dz\dtau.
 \end{equation}
 
 We next pass to the limit $n\to\infty$ in 
 \begin{align*}
  \int_{\0^M}  \sum_{i=1}^n\big(w(t_i)\eta(t_{i-1})-w(t_{i-1})\eta(t_{i})\big)\dz
  &= \int_{\0^M} w(t)\eta(t_{n-1})-u_0\eta(t_1)\dz \\[1ex]
  &\quad+\int_{\0^M} \sum_{i=1}^{n-1} w(t_i)(\eta(t_{i-1})-\eta(t_{i+1}) )\dz.
 \end{align*}
Recalling \eqref{cv4}, we have that $w\in{\rm  C}([0,t], H^2_0(\0^M)')$, and therefore
 \begin{align}\label{ref2}
 \int_{\0^M} w(t)\eta(t_{n-1})-u_0\eta(t_1)\dz \to \int_{\0^M} w(t)\eta(t)-u_0\eta(0)\dz \qquad\text{as $n\to\infty$.}
 \end{align}
 Furthermore, exploiting that $w\in{\rm  C}([0,t], H^2_0(\0^M)')$, we also have
 \begin{align*}
 &\hspace{-1cm}\sum_{i=1}^{n-1}\int_{\0^M} w(t_i)(\eta(t_{i+1})- \eta(t_{i-1}))\dz\\[1ex]
 & =\sum_{i=1}^{n-1}\int_{\0^M} w(t_i)(\eta(t_{i+1})- \eta(t_{i}))\dz+ \sum_{i=1}^{n-1}\int_{\0^M} w(t_i)(\eta(t_{i})- \eta(t_{i-1}))\dz\\[1ex]
 &\to2\int_0^t\int_{\0^M}w\p_t\eta\dz\dtau \qquad\text{as $n\to\infty$.}
 \end{align*}
The latter convergence, together with \eqref{ref1}--\eqref{ref2}, yields from \eqref{1d} the integral identity
 \begin{equation}\label{2d}
 \begin{aligned}
 &\hspace{-1cm}\int_{\0} w(t) \eta(t)\dz-\int_{\0} u_0 \eta(0)\dz-\int_0^t\int_{\0} w \p_t \eta\dz\dtau\\[1ex]
 &+\int_0^t\int_{\0}  \mu\nabla w\cdot\nabla \eta +J_1\p_y\eta+J_2\eta \dz\dtau=0.
 \end{aligned}
 \end{equation}
 Setting $u(t,z)\coloneqq e^{\gamma t} w(t,z)$ and $\phi(t,z)\coloneqq e^{-\gamma t}\eta(t,z)$ for $t\geq 0$ and $z\in\0$, it follows from Theorem~\ref{MT2} and \eqref{2d} that $u$ 
 is a weak solution to \eqref{PB'} as stated in Theorem~\ref{MT3}~(i) and~(ii).
 Finally, the energy estimate ~\eqref{EnE} follows directly from~\eqref{EnEs}.

\section*{Acknowledgments} 
The authors are grateful to Helmut Abels for interesting discussions on the subject of this paper. 
This work is partially supported by the RTG 2339 ``Interfaces, Complex Structures, and Singular Limits'' of the German Science Foundation (DFG) and the grant Z 387-N of the Austrian Science Fund (FWF), Austria.
 
\bibliographystyle{siam}
\bibliography{Literature}
\end{document}